\documentclass[11pt]{article}\UseRawInputEncoding
\usepackage{amssymb,amsfonts,amsmath,latexsym,epsf,tikz,url}
\usepackage[usenames,dvipsnames]{pstricks}
\usepackage{pstricks-add}
\usepackage{epsfig}
\usepackage{pst-grad} 
\usepackage{pst-plot} 
\usepackage[space]{grffile} 
\usepackage{etoolbox} 
\makeatletter 
\patchcmd\Gread@eps{\@inputcheck#1 }{\@inputcheck"#1"\relax}{}{}
\makeatother

\newtheorem{theorem}{Theorem}[section]

\newtheorem{corollary}[theorem]{Corollary}

\newtheorem{remark}[theorem]{Remark}
\newtheorem{example}[theorem]{Example}

\newtheorem{problem}[theorem]{Problem}
\newcommand{\proof}{\noindent{\bf Proof.\ }}
\newcommand{\qed}{\hfill $\square$\medskip}

\begin{document}

\title{Strong domination number of some operations on a graph}

\author{
Saeid Alikhani$^{1,}$\footnote{Corresponding author}  
\and
Nima Ghanbari$^{2}$
\and
Hassan Zaherifar$^{1}$
}

\date{\today}

\maketitle

\begin{center}
$^{1}$Department of Mathematical Sciences, Yazd University, 89195-741, Yazd, Iran\\
$^2$Department of Informatics, University of Bergen, P.O. Box 7803, 5020 Bergen, Norway\\
{\tt alikhani@yazd.ac.ir ~~ Nima.Ghanbari@uib.no ~~ hzaherifar@gmail.com}
\end{center}

\begin{abstract}
 Let $G=(V(G),E(G))$ be a simple graph. A set $D\subseteq V(G)$ is a strong dominating set of $G$, if for every vertex $x\in V(G)\setminus D$ there is a vertex $y\in D$ with $xy\in E(G)$ and $deg(x)\leq deg(y)$. The strong domination number $\gamma_{st}(G)$ is defined as the minimum cardinality of a strong dominating set.  In this paper, we examine the effects on $\gamma_{st}(G)$ when $G$ is modified by operations on edge (or edges) of $G$.  
\end{abstract}


\medskip
\noindent{\bf AMS Subj.\ Class.:} 05C15, 05C25

\section{Introduction}

A dominating set of a graph $G=(V,E)$ is any subset $D$ of $V$ such that every vertex not in $D$ is adjacent to at least one member of $D$.  
The minimum cardinality of all dominating sets of $G$ is called  the  domination number of $G$ and is denoted by $\gamma(G)$. This parameter has  been extensively studied in the literature and there are  hundreds of papers concerned with domination.  
For a detailed treatment of domination theory, the reader is referred to \cite{domination}. Also, the concept of domination and related invariants have
been generalized in many ways.

The {corona product} $G\circ H$ of two graphs $G$ and $H$ is defined as the graph obtained by taking one copy of $G$ and $\vert V(G)\vert $ copies of $H$ and joining the $i$-th vertex of $G$ to every vertex in the $i$-th copy of $H$. 

 A set $D\subseteq V(G)$ is a strong dominating set of $G$, if for every vertex $x\in V(G)\setminus D$ there is a vertex $y\in D$ with $xy\in E(G)$ and $deg(x)\leq deg(y)$. The strong
domination number $\gamma_{st}(G)$ is defined as the minimum cardinality of a strong dominating set. 
A strong dominating set with cardinality $\gamma_{st}(G)$ is called a $\gamma_{st}$-set.
The strong domination number was introduced in \cite{DM} and some upper bounds on this parameter presented in \cite{DM2,DM}. Similar to strong domination number, a set $D\subset V$  is a weak  dominating set of $G$ if every vertex $v\in V\setminus S$  is
adjacent to a vertex $u\in D$ such that $deg(v)\geq deg(u)$ (see \cite{Boutrig}). The minimum cardinality of a weak dominating set of $G$ is denoted by $\gamma_w(G)$. Boutrig and  Chellali proved that the relation $\gamma_w(G)+\frac{3}{\Delta+1}\gamma_{st}(G)\leq n$ holds for any connected graph of order $n\geq 3.$

 Motivated by counting of the number of dominating sets of a graph and domination polynomial (see e.g. \cite{euro,saeid1}), recently, we have studied the number of the strong dominating sets for certain
graphs \cite{JAS}.  

Let $e$ be an edge of a connected simple graph $G$. The graph obtained by removing  an edge $e$ from $G$ is denoted by $G-e$.
The edge subdivision operation for an edge $\{u,v\}\in E$ is the deletion of $\{u,v\}$ from $G$ and the addition of two edges $\{u,w\}$ and $\{w,v\}$ along with the new vertex $w$. 
A graph which has been derived from $G$ by an edge subdivision operation for edge $e$ is denoted by $G_e$.  The $k$-subdivision of $G$, denoted by $G^{\frac{1}{k}}$, is constructed by replacing each edge $v_iv_j$ of $G$ with a path of length $k$. 
An edge contraction is an operation that removes an edge from a graph while simultaneously merging the two vertices that it previously joined. The resulting induced graph is written as $G/e$. 

\medskip 
 In the next section,   we examine the effects on $\gamma_{st}(G)$ when $G$ is modified by operations edge deletion, edge subdivision and edge contraction. Also we study the strong domination number of $k$-subdivision of $G$ in Section 3.

 \section{Strong domination number of some operations on a graph}
 
  In this section, we study the relations between the strong domination number of $G,G-e, G_e$ and $G/e$. First we consider the edge deletion.

 \subsection{Edge deletion}
 
 We begin with the following result:

 \begin{theorem}\label{edge-deletion}
 	Let $G=(V,E)$ be a graph which is not $K_2$, and $e=uv\in E$. Then,
 	$$\gamma_{st} (G)-1\leq \gamma_{st}(G-e)\leq \gamma_{st} (G)+\deg (u)+\deg (v)-2.$$	
 \end{theorem}
 
 \begin{proof}
 	First we find the upper bound for $\gamma_{st}(G-e)$. Suppose that $D$ is a strong dominating set of $G$. In the worst case, both vertices $u$ and $v$ are in $D$ and $u$ has the same degree as some of its neighbours (except $v$) and strong dominates them, and the same for $v$. Suppose that $u'$ is adjacent to $u$, $u'\neq v$, $\deg(u)=\deg(u')$, and $u'$ is strong dominated only by $u$. Then, by removing $e$, there is no vertex such that strong dominates $u'$. So, we remove $u$ from $D$ and put all of its neighbours in $D$. Now, $u$ is strong dominated by at least $u'$. We have the same argument for $v$ too. So, by considering 
 	$$D'=\left( D \cup N(u) \cup N(v) \right) \setminus \{u,v\},$$
 	in this case, we have a strong dominating set. 
 	If we can keep $u$ in our strong dominating set to strong dominate at least one vertex (say $u''$), but condition for $v$ be the same as before, then we consider 
 	$$D''=\left( D \cup N(u) \cup N(v) \right) \setminus \{u'',v\}, $$
 	and we are done. If we can keep $u$ in our strong dominating set to strong dominate at least one vertex (say $u'''$), and keep $v$ in our strong dominating set to strong dominate at least one vertex (say $v'''$), then we consider 
 	$$D'''=\left( D \cup N(u) \cup N(v) \right) \setminus \{u''',v'''\}, $$
 	and we have a strong dominating set. 
 	Hence, in all cases, we have  
 	$$\gamma_{st}(G-e)\leq \gamma_{st} (G)+\deg (u)+\deg (v)-2.$$
 	Note that if $u\in D$ and $v\notin D$, then after removing $e$, we just need add $v$ to $D$ and the inequality holds for this condition too. If $u,v\notin D$, then after removing $e$, they are strong dominated by the same vertices as before. 
 	
 	Now, we find a lower bound for $\gamma_{st}(G-e)$. First we remove $e$ and find a strong dominating set for $G-e$. Suppose that this set is $S$. We have the following cases:
 	\begin{itemize}
 		\item[(i)]
 		$u,v\in S$. In this case, adding edge $e$ does not make any difference and $S$ is a strong dominating set of $G$ too. So $\gamma_{st}(G)\leq \gamma_{st} (G-e)$.
 		\item[(ii)]
 		$u\in S$ and $v\notin S$. In this case, after adding edge $e$, let $S'=S\cup  \{v\}$. One can easily check that $S'$ is a strong dominating set of $G$, and  $\gamma_{st}(G)\leq \gamma_{st} (G-e)+1$.
 		
 		\item[(iii)]
 		$u,v\notin S$. Without loss of generality, suppose that $\deg (u) \leq \deg(v)$. After adding edge $e$, let $S''=S\cup  \{v\}$. Then, $u$ is strong dominated by $v$ and all other vertices in $V(G)\setminus S'$ are strong dominated as before. Hence, $S''$ is a strong dominating set of $G$, and  $\gamma_{st}(G)\leq \gamma_{st} (G-e)+1$.
 	\end{itemize}
 	Therefore in all cases we have $\gamma_{st}(G-e)\geq \gamma_{st} (G)-1$, and  we have the result. 	\qed
 \end{proof}

 \begin{figure}
 	\begin{center}
 		\psscalebox{0.55 0.55}
 		{
 			\begin{pspicture}(0,-7.6000004)(13.202778,4.8027782)
 			\psline[linecolor=black, linewidth=0.08](5.8013887,-0.99861085)(7.4013886,-0.99861085)(7.4013886,-0.99861085)
 			\psline[linecolor=black, linewidth=0.08](7.4013886,-0.99861085)(8.601389,1.4013891)(8.601389,1.4013891)
 			\psline[linecolor=black, linewidth=0.08](7.4013886,-0.99861085)(8.601389,-0.99861085)(8.601389,-0.99861085)
 			\psline[linecolor=black, linewidth=0.08](7.4013886,-0.99861085)(8.601389,-3.3986108)(8.601389,-3.3986108)
 			\psline[linecolor=black, linewidth=0.08](8.601389,-0.99861085)(9.801389,-0.19861084)(9.801389,-0.19861084)
 			\psline[linecolor=black, linewidth=0.08](8.601389,-0.99861085)(9.801389,-0.99861085)(9.801389,-0.99861085)
 			\psline[linecolor=black, linewidth=0.08](8.601389,-0.99861085)(9.801389,-1.7986108)(9.801389,-1.7986108)
 			\psline[linecolor=black, linewidth=0.08](8.601389,-3.3986108)(9.801389,-2.5986109)(9.801389,-2.5986109)
 			\psline[linecolor=black, linewidth=0.08](8.601389,-3.3986108)(9.801389,-3.3986108)(9.801389,-3.3986108)
 			\psline[linecolor=black, linewidth=0.08](8.601389,-3.3986108)(9.801389,-4.198611)(9.801389,-4.198611)
 			\psline[linecolor=black, linewidth=0.08](8.601389,1.4013891)(9.801389,1.4013891)(9.801389,1.4013891)
 			\psline[linecolor=black, linewidth=0.08](8.601389,1.4013891)(9.801389,2.201389)(9.801389,2.201389)
 			\psline[linecolor=black, linewidth=0.08](8.601389,1.4013891)(9.801389,0.60138917)(9.801389,0.60138917)
 			\psline[linecolor=black, linewidth=0.08](9.801389,2.201389)(13.001389,2.201389)(13.001389,2.201389)
 			\psline[linecolor=black, linewidth=0.08](9.801389,1.4013891)(13.001389,1.4013891)(13.001389,1.4013891)
 			\psline[linecolor=black, linewidth=0.08](9.801389,0.60138917)(13.001389,0.60138917)(13.001389,0.60138917)
 			\psline[linecolor=black, linewidth=0.08](9.801389,-0.19861084)(13.001389,-0.19861084)(13.001389,-0.19861084)
 			\psline[linecolor=black, linewidth=0.08](9.801389,-0.99861085)(13.001389,-0.99861085)(13.001389,-0.99861085)
 			\psline[linecolor=black, linewidth=0.08](9.801389,-1.7986108)(13.001389,-1.7986108)(13.001389,-1.7986108)
 			\psline[linecolor=black, linewidth=0.08](9.801389,-2.5986109)(13.001389,-2.5986109)(13.001389,-2.5986109)
 			\psline[linecolor=black, linewidth=0.08](9.801389,-3.3986108)(13.001389,-3.3986108)(13.001389,-3.3986108)
 			\psline[linecolor=black, linewidth=0.08](9.801389,-4.198611)(13.001389,-4.198611)(13.001389,-4.198611)
 			\psline[linecolor=black, linewidth=0.08](4.601389,3.4013891)(3.401389,3.0013893)(3.401389,3.0013893)
 			\psline[linecolor=black, linewidth=0.08](4.601389,3.4013891)(3.401389,3.8013892)(3.401389,3.8013892)
 			\psline[linecolor=black, linewidth=0.08](4.601389,3.4013891)(3.401389,4.601389)(3.401389,4.601389)
 			\psline[linecolor=black, linewidth=0.08](4.601389,3.4013891)(3.401389,2.201389)(3.401389,2.201389)
 			\psline[linecolor=black, linewidth=0.08](4.601389,0.20138916)(3.401389,-0.19861084)(3.401389,-0.19861084)
 			\psline[linecolor=black, linewidth=0.08](4.601389,0.20138916)(3.401389,0.60138917)(3.401389,0.60138917)
 			\psline[linecolor=black, linewidth=0.08](4.601389,0.20138916)(3.401389,1.4013891)(3.401389,1.4013891)
 			\psline[linecolor=black, linewidth=0.08](4.601389,0.20138916)(3.401389,-0.99861085)(3.401389,-0.99861085)
 			\psline[linecolor=black, linewidth=0.08](4.601389,-2.9986107)(3.401389,-3.3986108)(3.401389,-3.3986108)
 			\psline[linecolor=black, linewidth=0.08](4.601389,-2.9986107)(3.401389,-2.5986109)(3.401389,-2.5986109)
 			\psline[linecolor=black, linewidth=0.08](4.601389,-2.9986107)(3.401389,-1.7986108)(3.401389,-1.7986108)
 			\psline[linecolor=black, linewidth=0.08](4.601389,-2.9986107)(3.401389,-4.198611)(3.401389,-4.198611)
 			\psline[linecolor=black, linewidth=0.08](4.601389,-6.198611)(3.401389,-6.598611)(3.401389,-6.598611)
 			\psline[linecolor=black, linewidth=0.08](4.601389,-6.198611)(3.401389,-5.7986107)(3.401389,-5.7986107)
 			\psline[linecolor=black, linewidth=0.08](4.601389,-6.198611)(3.401389,-4.998611)(3.401389,-4.998611)
 			\psline[linecolor=black, linewidth=0.08](4.601389,-6.198611)(3.401389,-7.398611)(3.401389,-7.398611)
 			\psline[linecolor=black, linewidth=0.08](5.8013887,-0.99861085)(4.601389,3.4013891)(4.601389,3.4013891)
 			\psline[linecolor=black, linewidth=0.08](5.8013887,-0.99861085)(4.601389,0.20138916)(4.601389,0.20138916)
 			\psline[linecolor=black, linewidth=0.08](5.8013887,-0.99861085)(4.601389,-2.9986107)(4.601389,-2.9986107)
 			\psline[linecolor=black, linewidth=0.08](5.8013887,-0.99861085)(4.601389,-6.198611)(4.601389,-6.198611)
 			\psline[linecolor=black, linewidth=0.08](0.20138885,4.601389)(3.401389,4.601389)(3.401389,4.601389)
 			\psline[linecolor=black, linewidth=0.08](0.20138885,3.8013892)(3.401389,3.8013892)(3.401389,3.8013892)
 			\psline[linecolor=black, linewidth=0.08](0.20138885,3.0013893)(3.401389,3.0013893)(3.401389,3.0013893)
 			\psline[linecolor=black, linewidth=0.08](0.20138885,1.4013891)(3.401389,1.4013891)(3.401389,1.4013891)
 			\psline[linecolor=black, linewidth=0.08](0.20138885,0.60138917)(3.401389,0.60138917)(3.401389,0.60138917)
 			\psline[linecolor=black, linewidth=0.08](0.20138885,-0.19861084)(3.401389,-0.19861084)(3.401389,-0.19861084)
 			\psline[linecolor=black, linewidth=0.08](0.20138885,-1.7986108)(3.401389,-1.7986108)(3.401389,-1.7986108)
 			\psline[linecolor=black, linewidth=0.08](0.20138885,-2.5986109)(3.401389,-2.5986109)(3.401389,-2.5986109)
 			\psline[linecolor=black, linewidth=0.08](0.20138885,-3.3986108)(3.401389,-3.3986108)(3.401389,-3.3986108)
 			\psline[linecolor=black, linewidth=0.08](0.20138885,-4.998611)(3.401389,-4.998611)(3.401389,-4.998611)
 			\psline[linecolor=black, linewidth=0.08](0.20138885,-5.7986107)(3.401389,-5.7986107)(3.401389,-5.7986107)
 			\psline[linecolor=black, linewidth=0.08](0.20138885,-6.598611)(3.401389,-6.598611)(3.401389,-6.598611)
 			\psline[linecolor=black, linewidth=0.08](0.20138885,-7.398611)(3.401389,-7.398611)(3.401389,-7.398611)
 			\psline[linecolor=black, linewidth=0.08](0.20138885,2.201389)(3.401389,2.201389)(3.401389,2.201389)
 			\psline[linecolor=black, linewidth=0.08](0.20138885,-0.99861085)(3.401389,-0.99861085)(3.401389,-0.99861085)
 			\psline[linecolor=black, linewidth=0.08](0.20138885,-4.198611)(3.401389,-4.198611)(3.401389,-4.198611)
 			\psdots[linecolor=black, dotsize=0.4](5.8013887,-0.99861085)
 			\psdots[linecolor=black, dotsize=0.4](7.4013886,-0.99861085)
 			\psdots[linecolor=black, dotsize=0.4](11.401389,2.201389)
 			\psdots[linecolor=black, dotsize=0.4](11.401389,1.4013891)
 			\psdots[linecolor=black, dotsize=0.4](11.401389,0.60138917)
 			\psdots[linecolor=black, dotsize=0.4](11.401389,-0.19861084)
 			\psdots[linecolor=black, dotsize=0.4](11.401389,-0.99861085)
 			\psdots[linecolor=black, dotsize=0.4](11.401389,-1.7986108)
 			\psdots[linecolor=black, dotsize=0.4](11.401389,-2.5986109)
 			\psdots[linecolor=black, dotsize=0.4](11.401389,-3.3986108)
 			\psdots[linecolor=black, dotsize=0.4](11.401389,-4.198611)
 			\psdots[linecolor=black, dotsize=0.4](1.8013889,4.601389)
 			\psdots[linecolor=black, dotsize=0.4](1.8013889,3.8013892)
 			\psdots[linecolor=black, dotsize=0.4](1.8013889,3.0013893)
 			\psdots[linecolor=black, dotsize=0.4](1.8013889,2.201389)
 			\psdots[linecolor=black, dotsize=0.4](1.8013889,1.4013891)
 			\psdots[linecolor=black, dotsize=0.4](1.8013889,0.60138917)
 			\psdots[linecolor=black, dotsize=0.4](1.8013889,-0.19861084)
 			\psdots[linecolor=black, dotsize=0.4](1.8013889,-0.99861085)
 			\psdots[linecolor=black, dotsize=0.4](1.8013889,-1.7986108)
 			\psdots[linecolor=black, dotsize=0.4](1.8013889,-2.5986109)
 			\psdots[linecolor=black, dotsize=0.4](1.8013889,-3.3986108)
 			\psdots[linecolor=black, dotsize=0.4](1.8013889,-4.198611)
 			\psdots[linecolor=black, dotsize=0.4](1.8013889,-4.998611)
 			\psdots[linecolor=black, dotsize=0.4](1.8013889,-5.7986107)
 			\psdots[linecolor=black, dotsize=0.4](1.8013889,-6.598611)
 			\psdots[linecolor=black, dotsize=0.4](1.8013889,-7.398611)
 			\psdots[linecolor=black, dotstyle=o, dotsize=0.4, fillcolor=white](3.401389,4.601389)
 			\psdots[linecolor=black, dotstyle=o, dotsize=0.4, fillcolor=white](3.401389,3.8013892)
 			\psdots[linecolor=black, dotstyle=o, dotsize=0.4, fillcolor=white](3.401389,3.0013893)
 			\psdots[linecolor=black, dotstyle=o, dotsize=0.4, fillcolor=white](3.401389,2.201389)
 			\psdots[linecolor=black, dotstyle=o, dotsize=0.4, fillcolor=white](3.401389,1.4013891)
 			\psdots[linecolor=black, dotstyle=o, dotsize=0.4, fillcolor=white](3.401389,0.60138917)
 			\psdots[linecolor=black, dotstyle=o, dotsize=0.4, fillcolor=white](3.401389,-0.19861084)
 			\psdots[linecolor=black, dotstyle=o, dotsize=0.4, fillcolor=white](3.401389,-0.99861085)
 			\psdots[linecolor=black, dotstyle=o, dotsize=0.4, fillcolor=white](3.401389,-1.7986108)
 			\psdots[linecolor=black, dotstyle=o, dotsize=0.4, fillcolor=white](3.401389,-2.5986109)
 			\psdots[linecolor=black, dotstyle=o, dotsize=0.4, fillcolor=white](3.401389,-3.3986108)
 			\psdots[linecolor=black, dotstyle=o, dotsize=0.4, fillcolor=white](3.401389,-4.998611)
 			\psdots[linecolor=black, dotstyle=o, dotsize=0.4, fillcolor=white](3.401389,-5.7986107)
 			\psdots[linecolor=black, dotstyle=o, dotsize=0.4, fillcolor=white](3.401389,-6.598611)
 			\psdots[linecolor=black, dotstyle=o, dotsize=0.4, fillcolor=white](3.401389,-7.398611)
 			\psdots[linecolor=black, dotstyle=o, dotsize=0.4, fillcolor=white](0.20138885,4.601389)
 			\psdots[linecolor=black, dotstyle=o, dotsize=0.4, fillcolor=white](0.20138885,3.8013892)
 			\psdots[linecolor=black, dotstyle=o, dotsize=0.4, fillcolor=white](0.20138885,3.0013893)
 			\psdots[linecolor=black, dotstyle=o, dotsize=0.4, fillcolor=white](0.20138885,2.201389)
 			\psdots[linecolor=black, dotstyle=o, dotsize=0.4, fillcolor=white](0.20138885,1.4013891)
 			\psdots[linecolor=black, dotstyle=o, dotsize=0.4, fillcolor=white](0.20138885,0.60138917)
 			\psdots[linecolor=black, dotstyle=o, dotsize=0.4, fillcolor=white](0.20138885,-0.19861084)
 			\psdots[linecolor=black, dotstyle=o, dotsize=0.4, fillcolor=white](0.20138885,-0.99861085)
 			\psdots[linecolor=black, dotstyle=o, dotsize=0.4, fillcolor=white](0.20138885,-1.7986108)
 			\psdots[linecolor=black, dotstyle=o, dotsize=0.4, fillcolor=white](0.20138885,-2.5986109)
 			\psdots[linecolor=black, dotstyle=o, dotsize=0.4, fillcolor=white](0.20138885,-3.3986108)
 			\psdots[linecolor=black, dotstyle=o, dotsize=0.4, fillcolor=white](0.20138885,-4.198611)
 			\psdots[linecolor=black, dotstyle=o, dotsize=0.4, fillcolor=white](0.20138885,-4.998611)
 			\psdots[linecolor=black, dotstyle=o, dotsize=0.4, fillcolor=white](0.20138885,-5.7986107)
 			\psdots[linecolor=black, dotstyle=o, dotsize=0.4, fillcolor=white](0.20138885,-6.598611)
 			\psdots[linecolor=black, dotstyle=o, dotsize=0.4, fillcolor=white](0.20138885,-7.398611)
 			\psdots[linecolor=black, dotstyle=o, dotsize=0.4, fillcolor=white](4.601389,3.4013891)
 			\psdots[linecolor=black, dotstyle=o, dotsize=0.4, fillcolor=white](4.601389,0.20138916)
 			\psdots[linecolor=black, dotstyle=o, dotsize=0.4, fillcolor=white](4.601389,-2.9986107)
 			\psdots[linecolor=black, dotstyle=o, dotsize=0.4, fillcolor=white](4.601389,-6.198611)
 			\psdots[linecolor=black, dotstyle=o, dotsize=0.4, fillcolor=white](8.601389,1.4013891)
 			\psdots[linecolor=black, dotstyle=o, dotsize=0.4, fillcolor=white](8.601389,-0.99861085)
 			\psdots[linecolor=black, dotstyle=o, dotsize=0.4, fillcolor=white](8.601389,-3.3986108)
 			\psdots[linecolor=black, dotstyle=o, dotsize=0.4, fillcolor=white](9.801389,2.201389)
 			\psdots[linecolor=black, dotstyle=o, dotsize=0.4, fillcolor=white](9.801389,1.4013891)
 			\psdots[linecolor=black, dotstyle=o, dotsize=0.4, fillcolor=white](9.801389,0.60138917)
 			\psdots[linecolor=black, dotstyle=o, dotsize=0.4, fillcolor=white](9.801389,-0.19861084)
 			\psdots[linecolor=black, dotstyle=o, dotsize=0.4, fillcolor=white](9.801389,-0.99861085)
 			\psdots[linecolor=black, dotstyle=o, dotsize=0.4, fillcolor=white](9.801389,-1.7986108)
 			\psdots[linecolor=black, dotstyle=o, dotsize=0.4, fillcolor=white](9.801389,-2.5986109)
 			\psdots[linecolor=black, dotstyle=o, dotsize=0.4, fillcolor=white](9.801389,-3.3986108)
 			\psdots[linecolor=black, dotstyle=o, dotsize=0.4, fillcolor=white](9.801389,-4.198611)
 			\psdots[linecolor=black, dotstyle=o, dotsize=0.4, fillcolor=white](13.001389,2.201389)
 			\psdots[linecolor=black, dotstyle=o, dotsize=0.4, fillcolor=white](13.001389,1.4013891)
 			\psdots[linecolor=black, dotstyle=o, dotsize=0.4, fillcolor=white](13.001389,0.60138917)
 			\psdots[linecolor=black, dotstyle=o, dotsize=0.4, fillcolor=white](13.001389,-0.19861084)
 			\psdots[linecolor=black, dotstyle=o, dotsize=0.4, fillcolor=white](13.001389,-0.99861085)
 			\psdots[linecolor=black, dotstyle=o, dotsize=0.4, fillcolor=white](13.001389,-1.7986108)
 			\psdots[linecolor=black, dotstyle=o, dotsize=0.4, fillcolor=white](13.001389,-2.5986109)
 			\psdots[linecolor=black, dotstyle=o, dotsize=0.4, fillcolor=white](13.001389,-3.3986108)
 			\psdots[linecolor=black, dotstyle=o, dotsize=0.4, fillcolor=white](13.001389,-4.198611)
 			\rput[bl](6.541389,-0.69861084){$e$}
 			\rput[bl](7.181389,-1.5186108){$u$}
 			\rput[bl](5.8613887,-1.4986109){$v$}
 			\rput[bl](8.421389,1.7213892){$u_1$}
 			\rput[bl](8.441389,-0.67861086){$u_2$}
 			\rput[bl](8.441389,-4.018611){$u_3$}
 			\rput[bl](4.501389,3.8613892){$v_1$}
 			\rput[bl](4.481389,0.64138913){$v_2$}
 			\rput[bl](4.441389,-3.6386108){$v_3$}
 			\rput[bl](4.481389,-6.8186107){$v_4$}
 			\psdots[linecolor=black, dotstyle=o, dotsize=0.4, fillcolor=white](3.401389,-4.198611)
 			\end{pspicture}
 		}
 	\end{center}
 	\caption{Graph $G$} \label{g-e-upper}
 \end{figure}

 \begin{remark}
 	{\normalfont
 		Bounds in Theorem \ref{edge-deletion} are tight. For the upper bound, consider $G$ as shown in Figure \ref{g-e-upper}. One can easily check that the set of black vertices is a strong dominating set of $G$ (say $D$). If we remove edge $e$, then for example, for the vertex $v_1$, we have $\deg (v) < \deg (v_1)$, and $v$ does not strong dominate $v_1$ any more. Since all of the neighbours of $v_1$ have less degree, so we should have it in our strong dominating set. So, by the same argument for all vertices,
 		$$D'=\left( D \cup \{v_1,v_2,v_3,v_4,u_1,u_2,u_3 \} \right) \setminus \{v,u\} $$
 		is a strong dominating set for $G-e$, and we are done. For the lower bound, consider $H$ as shown in Figure \ref{g-e-lower}. One can easily check that 
 		$S=\{v_1,v_2,v_3,u_1,u_2,u_3,u_4 \}$
 		is a strong dominating set for $H-e$, and $S'=\{u,v_1,v_2,v_3,u_1,u_2,u_3,u_4 \}$ is a strong dominating set for $H$, as desired.
 	}
 \end{remark}

 \begin{remark}
 	\normalfont
 	It is easy to see that if $P_n$ and $C_n$ are the path graph and the cycle graph of order $n$, respectively, then $\gamma_{st}(P_{n})=\gamma_{st}(C_{n})=\lceil\frac{n}{3}\rceil.$
 	So the path $P_n$ (if $n\not\equiv 1$ $(mod\, 3)$ and $e$ is edge incident with leaves), is  another examples for the tightness of the upper bound in  Theorem \ref{edge-deletion}. 
 	Note that we do not have equalities of  Theorem \ref{edge-deletion} for the cycles. 
 	
 \end{remark}

 We close this subsection with the following theorem which is about the strong domination number of corona of two graphs $G_1\circ G_2$ when it is modified by deletion of an edge.

 \begin{figure}
 	\begin{center}
 		\psscalebox{0.55 0.55}
 		{
 			\begin{pspicture}(0,-7.4314423)(11.594231,-1.7743268)
 			\psline[linecolor=black, linewidth=0.08](4.9971156,-4.3714423)(6.5971155,-4.3714423)(6.5971155,-4.3714423)
 			\psdots[linecolor=black, dotsize=0.4](4.9971156,-4.3714423)
 			\psdots[linecolor=black, dotsize=0.4](6.5971155,-4.3714423)
 			\rput[bl](5.7371154,-4.071442){$e$}
 			\rput[bl](6.3771152,-4.8914423){$u$}
 			\rput[bl](5.0571156,-4.8714423){$v$}
 			\rput[bl](7.9771156,-2.3714423){$u_1$}
 			\rput[bl](9.6371155,-4.9114423){$u_2$}
 			\rput[bl](9.6371155,-6.111442){$u_3$}
 			\rput[bl](3.3171155,-2.4314423){$v_1$}
 			\rput[bl](1.5971155,-4.971442){$v_2$}
 			\rput[bl](1.6371155,-6.551442){$v_3$}
 			\psline[linecolor=black, linewidth=0.08](6.5971155,-4.3714423)(8.197116,-2.7714422)(8.197116,-2.7714422)
 			\psline[linecolor=black, linewidth=0.08](8.197116,-2.7714422)(9.797115,-1.9714422)(9.797115,-1.9714422)
 			\psline[linecolor=black, linewidth=0.08](8.197116,-2.7714422)(9.797115,-2.7714422)(9.797115,-2.7714422)
 			\psline[linecolor=black, linewidth=0.08](8.197116,-2.7714422)(9.797115,-3.5714424)(9.797115,-3.5714424)
 			\psline[linecolor=black, linewidth=0.08](6.5971155,-4.3714423)(8.197116,-4.3714423)(9.797115,-4.3714423)(11.397116,-4.3714423)(11.397116,-4.3714423)
 			\psline[linecolor=black, linewidth=0.08](4.9971156,-4.3714423)(3.3971155,-4.3714423)(3.3971155,-4.3714423)
 			\psline[linecolor=black, linewidth=0.08](4.9971156,-4.3714423)(3.3971155,-2.7714422)(3.3971155,-2.7714422)
 			\psline[linecolor=black, linewidth=0.08](4.9971156,-4.3714423)(3.3971155,-5.971442)(3.3971155,-5.971442)
 			\psdots[linecolor=black, dotsize=0.4](8.197116,-2.7714422)
 			\psdots[linecolor=black, dotsize=0.4](9.797115,-1.9714422)
 			\psdots[linecolor=black, dotsize=0.4](9.797115,-2.7714422)
 			\psdots[linecolor=black, dotsize=0.4](9.797115,-3.5714424)
 			\psdots[linecolor=black, dotsize=0.4](8.197116,-4.3714423)
 			\psdots[linecolor=black, dotsize=0.4](9.797115,-4.3714423)
 			\psline[linecolor=black, linewidth=0.08](6.5971155,-4.3714423)(8.197116,-5.571442)(9.797115,-5.571442)(11.397116,-5.571442)(11.397116,-5.571442)
 			\psline[linecolor=black, linewidth=0.08](6.5971155,-4.3714423)(8.197116,-6.7714424)(9.797115,-6.7714424)(11.397116,-6.7714424)(11.397116,-6.7714424)
 			\psdots[linecolor=black, dotsize=0.4](11.397116,-4.3714423)
 			\psdots[linecolor=black, dotsize=0.4](8.197116,-5.571442)
 			\psdots[linecolor=black, dotsize=0.4](9.797115,-5.571442)
 			\psdots[linecolor=black, dotsize=0.4](11.397116,-5.571442)
 			\psdots[linecolor=black, dotsize=0.4](8.197116,-6.7714424)
 			\psdots[linecolor=black, dotsize=0.4](9.797115,-6.7714424)
 			\psdots[linecolor=black, dotsize=0.4](11.397116,-6.7714424)
 			\psdots[linecolor=black, dotsize=0.4](3.3971155,-2.7714422)
 			\psdots[linecolor=black, dotsize=0.4](3.3971155,-4.3714423)
 			\psdots[linecolor=black, dotsize=0.4](3.3971155,-5.971442)
 			\psdots[linecolor=black, dotsize=0.4](1.7971154,-1.9714422)
 			\psdots[linecolor=black, dotsize=0.4](1.7971154,-3.5714424)
 			\psdots[linecolor=black, dotsize=0.4](1.7971154,-4.3714423)
 			\psdots[linecolor=black, dotsize=0.4](0.19711548,-4.3714423)
 			\psdots[linecolor=black, dotsize=0.4](1.7971154,-5.971442)
 			\psdots[linecolor=black, dotsize=0.4](0.19711548,-5.971442)
 			\psline[linecolor=black, linewidth=0.08](1.7971154,-3.5714424)(3.3971155,-2.7714422)(1.7971154,-1.9714422)(1.7971154,-1.9714422)
 			\psline[linecolor=black, linewidth=0.08](0.19711548,-4.3714423)(1.7971154,-4.3714423)(3.3971155,-4.3714423)(3.3971155,-4.3714423)
 			\psline[linecolor=black, linewidth=0.08](0.19711548,-5.971442)(1.7971154,-5.971442)(3.3971155,-5.971442)(3.3971155,-5.971442)
 			\rput[bl](9.657116,-7.4314423){$u_4$}
 			\end{pspicture}
 		}
 	\end{center}
 	\caption{Graph $H$} \label{g-e-lower}
 \end{figure}

  \begin{theorem}
  	If $G_1$ and $G_2$ are two graphs, then 
  	\begin{equation*}
  	\gamma_{st}((G_1\circ G_2)-e)=\left\{
  	\begin{array}{ll}
  	\gamma_{st}(G_1\circ G_2)  &\quad\mbox{if  $e\in E((G_1)$ or $e \in E(G_2),$}\\
  	\gamma_{st}(G_1\circ G_2)+1   &\quad\mbox{if  $e=v_iv_j, v_i \in V(G_1), v_j\in V(G_2)$}.
  	\end{array}\right.
  	\end{equation*}
  \end{theorem}
  \begin{proof} 
  	In the removing edge $e$ of $G_1\circ G_2$, we have three cases:
  	
  	Case 1.  $e\in E(G_1)$. Since the minimum strong dominating set of $G_1\circ G_2$ is $V(G_1)$, so in this case, $\gamma_{st}((G_1\circ G_2)-e)=\gamma_{st}((G_1\circ G_2).$
  	
  	Case 2. $e\in E(G_2)$. In this case the minimum dominating set of $(G_1\circ G_2)-e$, does not change and so  $\gamma_{st}((G_1\circ G_2)-e)=\gamma_{st}((G_1\circ G_2)).$
  	
  	Case 3. If $e=uv$, $u\in V(G_1), v\in V(G_2)$ or $v\in V(G_1), u\in V(G_2)$. In this case by removing the edge $e$, one vertex of one copy of $V(G_2)$ does not dominate by the minimum strong dominating set of  $G_1\circ G_2$. Therefore $\gamma_{st}((G_1\circ G_2)-e)=\gamma_{st}(G_1\circ G_2)+1$. \qed
  	
  \end{proof}

 \subsection{Edge subdivision}

 In this subsection, we examine the  effects on $\gamma_{st}(G)$ when $G$ is modified by subdivision  on an edge of $G$.

 \begin{theorem}\label{edge-subdivision}
 	If $G=(V,E)$ is a graph and $e\in E$, then,
 	$$\gamma_{st} (G)\leq \gamma_{st}(G_e)\leq \gamma_{st} (G)+1.$$		
 \end{theorem}
 
 \begin{proof}
 	First we find the upper bound for $\gamma_{st}(G_e)$. Suppose that  $v_e$ is the new vertex in $G_e$ and also $D$ is a strong dominating set of $G$. One can easily check that $D'=D\cup\{v_e\}$ is a strong dominating set of $G_e$, and we are done. Now, we find the lower bound. 	Consider the graph  $G_e$ and let $D_e$ be its strong dominating set. If $v_e\in D_e$, then it may strong dominate its neighbours or not. If it does, then since its degree is $2$, its neighbours should have degree at most two. So for $G$, let strong dominating set be the old one by adding the neighbour of $v_e$ with higher (or equal) degree and removing $v_e$, and hence $\gamma_{st} (G)\leq \gamma_{st}(G_e)$. If it does not, then removing that from our strong dominating set does not have effect on being strong  dominating set for $G$. So  $\gamma_{st} (G)\leq \gamma_{st}(G_e)-1$. So, if $v_e\in D_e$, then $\gamma_{st} (G)\leq \gamma_{st}(G_e)$. if $v_e\notin D_e$, then one can easily check that $D_e$ is a strong dominating set of $G$ too. Therefore we have the result.
 	\qed
 \end{proof}

 \begin{remark}
 	{\normalfont
 		The bounds in Theorem \ref{edge-subdivision} are tight. For the upper bound, consider $G$ as the cycle graph $C_{3k}$ or the path graph $P_{3k}$. For the lower bound, consider $G$ as the cycle graph $C_{3k+1}$ or the path graph $P_{3k+1}$.
 	}
 \end{remark}

 \begin{remark}
 	{\normalfont
 		From Theorems \ref{edge-deletion} and \ref{edge-subdivision}, we see that for some graphs $\gamma_{st} (G-e)= \gamma_{st}(G_e)$. For example, the cycle graphs $C_n$ (when $n\not \equiv 0$ $(mod\, 3))$, and the complete bipartite graph $K_{m,n}$
 satisfy this equality.   
 The characterization of these kind of graphs is an interesting problem which we propose it here:
 		
 		\begin{problem} 
 		Characterize graph $G$ and edge $e$ with $\gamma_{st} (G-e)= \gamma_{st}(G_e)$. 
 		\end{problem} 
} \end{remark}

   The following theorem gives a relation  for  the strong domination number  of the corona product of two graphs when it is modified by subdivision of an edge. 
   
   \begin{theorem}
   	If $G_1$ and $G_2$ are two graphs, then 
   	\begin{equation*}
   	\gamma_{st}((G_1\circ G_2)_e)=\left\{
   	\begin{array}{ll}
   	\gamma_{st}(G_1\circ G_2)  &\quad\mbox{if  $e\in E(G_1)$},\\
   	\gamma_{st}(G_1\circ G_2)+1   &\quad\mbox{if  $e \in E(G_2)$ or $e=v_iv_j , v_i \in V(G_1), v_j\in V(G_2)$}.
   	\end{array}\right.
   	\end{equation*}
   \end{theorem}
   \begin{proof}
   	If  $e\in E(G_1)$, since the minimum strong dominating set of $G_1\circ G_2$ is $V(G_1)$, so by subdividing $e$, the minimum strong dominating set of $(G_1\circ G_2)_e$ is also  $V(G_1)$ and so  $ \gamma_{st}((G_1\circ G_2)_e)= \gamma_{st}(G_1\circ G_2).$ 
   	If  $e \in E(G_2)$ or $e=v_iv_j, v_i \in V(G_1), v_j\in V(G_2)$, by subdividing edge $e$, one vertex of one copy of $G_2$ or vertex that added to $G_1\circ G_2$, does not dominate by  the minimum strong dominating set of $G_1\circ G_2$. Therefore in this case 
   	$ \gamma_{st}((G_1\circ G_2)_e)= \gamma_{st}((G_1\circ G_2)+1.$ \qed
   \end{proof}

\subsection{Edge contraction}

 In this subsection, we examine the  effects on $\gamma_{st}(G)$ when $G$ is modified by contraction  on an edge of $G$.

 \begin{theorem}\label{edge-contraction}
 	If $G=(V,E)$ is a graph  which is not $K_2$, and $e=uv\in E$, then,
 	$$\gamma_{st} (G)-\deg(u)-\deg(v)+3\leq \gamma_{st}(G/e)\leq \gamma_{st} (G)+1.$$		
 \end{theorem}
 
 \begin{proof}
 	Suppose that  $w$ is the new vertex in $G/e$ by contraction of $e$ and replacement of that with $u$ and $v$. First we find the upper bound for $\gamma_{st}(G/e)$.  Suppose that $D$ is a strong dominating set of $G$.	If at least one of $u$ and $v$ be in $D$, then $D'=\left( D\cup\{w\}\right) \setminus \{u,v\}$ is a strong dominating set for $G/e$, since every vertices in $V(G)\setminus D$ are strong dominated by same vertices as before or possibly $w$. If $u,v\notin D$, then one can easily check that $D'=\left( D\cup\{w\}\right)$ is a strong dominating set for $G/e$, and therefore $\gamma_{st}(G/e)\leq \gamma_{st} (G)+1$. Now, we find the lower bound for $\gamma_{st}(G/e)$. First, we find a strong dominating set $S$ for $G/e$. We have two cases:
 	\begin{itemize}
 		\item[(i)]
 		$w\notin S$. It is clear that after forming $G$, $S'=S\cup\{u\}$ is a strong dominating set of $G$ and we have $\gamma_{st}(G)\leq \gamma_{st} (G/e)+1$. 
 		
 		\item[(ii)]
 		$w\in S$. If every vertices in $V(G)\setminus S$ are strong dominating by vertices except $w$, then clearly $S'=\left(S\cup\{u,v\}\right)\setminus \{w\}$ is  a strong dominating set for $G$ and we have $\gamma_{st}(G)\leq \gamma_{st} (G/e)+1$. Now suppose that there exists $w'\in N(w)\setminus S$ such that $\deg(w')\leq\deg(w)$. 
 		We have the following cases:
 		\begin{itemize}
 			\item[(1)]
 			For all vertices  $x\in N(u)$, we have $\deg(x)\leq\deg(u)$, and for all vertices  $y\in N(v)$, we have $\deg(y)\leq\deg(v)$. In this case, one can easily check that
 			$$S'=\left( S \cup \{u,v\} \right) \setminus \{w\} $$
 			is a strong dominating set for $G$, and we have $\gamma_{st}(G)\leq \gamma_{st} (G/e)+1$.
 			\item[(2)]
 			For all vertices  $x\in N(u)$, we have $\deg(x)\leq\deg(u)$, and there exists  $y'\in N(v)$, such that $\deg(v)\leq\deg(y')$. In this case, let
 			$$S'=\left( S \cup  N(v) \right) \setminus \{w\}.$$
 			Then $v$ is strong dominated by $y'$ and the rest of vertices in $V(G)\setminus S$ are strong dominated as before (and possibly by $u$). So $S'$ is a strong dominating set, and hence $\gamma_{st}(G)\leq \gamma_{st} (G/e)+\deg(v)$.
 			\item[(3)]
 			There exists  $x'\in N(u)$, such that $\deg(u)\leq\deg(x')$, and there exists  $y'\in N(v)$, such that $\deg(v)\leq\deg(y')$.
 			In this case, let
 			$$S'=\left( S \cup \left( N(u)\setminus \{v\} \right) \cup \left( N(v)\setminus \{u\} \right) \right) \setminus \{w\}.$$
 			Then $u$ is strong dominated by $x'$, $v$ is strong dominated by $y'$, and the rest of vertices in $V(G)\setminus S$ are strong dominated as before. Hence $\gamma_{st}(G)\leq \gamma_{st} (G/e)+\deg(u)+\deg(v)-3$.
 		\end{itemize}
 		
 		Hence in any case, $\gamma_{st}(G/e) \geq \gamma_{st} (G)-\deg(u)-\deg(v)+3$.
 	\end{itemize}
 	Therefore we have the result. 
 	\qed
 \end{proof}

 \begin{figure}
 	\begin{center}
 		\psscalebox{0.56 0.56}
 		{
 			\begin{pspicture}(0,-7.9993057)(18.002779,0.40208343)
 			\psline[linecolor=black, linewidth=0.08](3.0013888,-2.9993055)(4.601389,-2.9993055)(4.601389,-2.9993055)
 			\rput[bl](3.7413888,-2.6993055){$e$}
 			\rput[bl](4.3813887,-3.5193055){$u$}
 			\rput[bl](3.061389,-3.4993055){$v$}
 			\psline[linecolor=black, linewidth=0.08](4.601389,-2.9993055)(6.201389,-2.9993055)(6.201389,-2.9993055)
 			\psline[linecolor=black, linewidth=0.08](6.201389,-0.5993054)(7.4013886,0.20069458)(7.4013886,0.20069458)
 			\psline[linecolor=black, linewidth=0.08](6.201389,-0.5993054)(7.4013886,-0.5993054)(7.4013886,-0.5993054)
 			\psline[linecolor=black, linewidth=0.08](6.201389,-0.5993054)(7.4013886,-1.3993055)(7.4013886,-1.3993055)
 			\psline[linecolor=black, linewidth=0.08](4.601389,-2.9993055)(6.201389,-5.3993053)(6.201389,-5.3993053)
 			\psline[linecolor=black, linewidth=0.08](4.601389,-2.9993055)(6.201389,-0.5993054)(6.201389,-0.5993054)
 			\psline[linecolor=black, linewidth=0.08](6.201389,-2.9993055)(7.4013886,-2.1993055)(7.4013886,-2.1993055)
 			\psline[linecolor=black, linewidth=0.08](6.201389,-2.9993055)(7.4013886,-2.9993055)(7.4013886,-2.9993055)
 			\psline[linecolor=black, linewidth=0.08](6.201389,-2.9993055)(7.4013886,-3.7993054)(7.4013886,-3.7993054)
 			\psline[linecolor=black, linewidth=0.08](6.201389,-5.3993053)(7.4013886,-4.5993056)(7.4013886,-4.5993056)
 			\psline[linecolor=black, linewidth=0.08](6.201389,-5.3993053)(7.4013886,-5.3993053)(7.4013886,-5.3993053)
 			\psline[linecolor=black, linewidth=0.08](6.201389,-5.3993053)(7.4013886,-6.1993055)(7.4013886,-6.1993055)
 			\psline[linecolor=black, linewidth=0.08](3.0013888,-2.9993055)(1.4013889,-2.9993055)(1.4013889,-2.9993055)
 			\psline[linecolor=black, linewidth=0.08](3.0013888,-2.9993055)(1.4013889,-0.5993054)
 			\psline[linecolor=black, linewidth=0.08](3.0013888,-2.9993055)(1.4013889,-5.3993053)(1.4013889,-5.3993053)
 			\psline[linecolor=black, linewidth=0.08](1.4013889,-0.5993054)(0.20138885,-0.5993054)(0.20138885,-0.5993054)
 			\psline[linecolor=black, linewidth=0.08](1.4013889,-0.5993054)(0.20138885,-1.3993055)(0.20138885,-1.3993055)
 			\psline[linecolor=black, linewidth=0.08](1.4013889,-0.5993054)(0.20138885,0.20069458)(0.20138885,0.20069458)
 			\psline[linecolor=black, linewidth=0.08](1.4013889,-2.9993055)(0.20138885,-2.1993055)(0.20138885,-2.1993055)
 			\psline[linecolor=black, linewidth=0.08](1.4013889,-2.9993055)(0.20138885,-2.9993055)(0.20138885,-2.9993055)
 			\psline[linecolor=black, linewidth=0.08](1.4013889,-2.9993055)(0.20138885,-3.7993054)(0.20138885,-3.7993054)
 			\psline[linecolor=black, linewidth=0.08](1.4013889,-5.3993053)(0.20138885,-4.5993056)(0.20138885,-4.5993056)
 			\psline[linecolor=black, linewidth=0.08](1.4013889,-5.3993053)(0.20138885,-5.3993053)(0.20138885,-5.3993053)
 			\psline[linecolor=black, linewidth=0.08](1.4013889,-5.3993053)(0.20138885,-6.1993055)(0.20138885,-6.1993055)
 			\psdots[linecolor=black, dotstyle=o, dotsize=0.4, fillcolor=white](7.4013886,0.20069458)
 			\psdots[linecolor=black, dotstyle=o, dotsize=0.4, fillcolor=white](7.4013886,-0.5993054)
 			\psdots[linecolor=black, dotstyle=o, dotsize=0.4, fillcolor=white](7.4013886,-1.3993055)
 			\psdots[linecolor=black, dotstyle=o, dotsize=0.4, fillcolor=white](7.4013886,-2.1993055)
 			\psdots[linecolor=black, dotstyle=o, dotsize=0.4, fillcolor=white](7.4013886,-2.9993055)
 			\psdots[linecolor=black, dotstyle=o, dotsize=0.4, fillcolor=white](7.4013886,-3.7993054)
 			\psdots[linecolor=black, dotstyle=o, dotsize=0.4, fillcolor=white](7.4013886,-4.5993056)
 			\psdots[linecolor=black, dotstyle=o, dotsize=0.4, fillcolor=white](7.4013886,-5.3993053)
 			\psdots[linecolor=black, dotstyle=o, dotsize=0.4, fillcolor=white](7.4013886,-6.1993055)
 			\psdots[linecolor=black, dotstyle=o, dotsize=0.4, fillcolor=white](0.20138885,0.20069458)
 			\psdots[linecolor=black, dotstyle=o, dotsize=0.4, fillcolor=white](0.20138885,-0.5993054)
 			\psdots[linecolor=black, dotstyle=o, dotsize=0.4, fillcolor=white](0.20138885,-1.3993055)
 			\psdots[linecolor=black, dotstyle=o, dotsize=0.4, fillcolor=white](0.20138885,-2.1993055)
 			\psdots[linecolor=black, dotstyle=o, dotsize=0.4, fillcolor=white](0.20138885,-2.9993055)
 			\psdots[linecolor=black, dotstyle=o, dotsize=0.4, fillcolor=white](0.20138885,-3.7993054)
 			\psdots[linecolor=black, dotstyle=o, dotsize=0.4, fillcolor=white](0.20138885,-4.5993056)
 			\psdots[linecolor=black, dotstyle=o, dotsize=0.4, fillcolor=white](0.20138885,-5.3993053)
 			\psdots[linecolor=black, dotstyle=o, dotsize=0.4, fillcolor=white](0.20138885,-6.1993055)
 			\psline[linecolor=black, linewidth=0.08](15.001389,-2.9993055)(16.601389,-2.9993055)(16.601389,-2.9993055)
 			\psline[linecolor=black, linewidth=0.08](16.601389,-0.5993054)(17.80139,0.20069458)(17.80139,0.20069458)
 			\psline[linecolor=black, linewidth=0.08](16.601389,-0.5993054)(17.80139,-0.5993054)(17.80139,-0.5993054)
 			\psline[linecolor=black, linewidth=0.08](16.601389,-0.5993054)(17.80139,-1.3993055)(17.80139,-1.3993055)
 			\psline[linecolor=black, linewidth=0.08](15.001389,-2.9993055)(16.601389,-5.3993053)(16.601389,-5.3993053)
 			\psline[linecolor=black, linewidth=0.08](15.001389,-2.9993055)(16.601389,-0.5993054)(16.601389,-0.5993054)
 			\psline[linecolor=black, linewidth=0.08](16.601389,-2.9993055)(17.80139,-2.1993055)(17.80139,-2.1993055)
 			\psline[linecolor=black, linewidth=0.08](16.601389,-2.9993055)(17.80139,-2.9993055)(17.80139,-2.9993055)
 			\psline[linecolor=black, linewidth=0.08](16.601389,-2.9993055)(17.80139,-3.7993054)(17.80139,-3.7993054)
 			\psline[linecolor=black, linewidth=0.08](16.601389,-5.3993053)(17.80139,-4.5993056)(17.80139,-4.5993056)
 			\psline[linecolor=black, linewidth=0.08](16.601389,-5.3993053)(17.80139,-5.3993053)(17.80139,-5.3993053)
 			\psline[linecolor=black, linewidth=0.08](16.601389,-5.3993053)(17.80139,-6.1993055)(17.80139,-6.1993055)
 			\psline[linecolor=black, linewidth=0.08](15.001389,-2.9993055)(13.401389,-2.9993055)(13.401389,-2.9993055)
 			\psline[linecolor=black, linewidth=0.08](15.001389,-2.9993055)(13.401389,-0.5993054)
 			\psline[linecolor=black, linewidth=0.08](15.001389,-2.9993055)(13.401389,-5.3993053)(13.401389,-5.3993053)
 			\psline[linecolor=black, linewidth=0.08](13.401389,-0.5993054)(12.201389,-0.5993054)(12.201389,-0.5993054)
 			\psline[linecolor=black, linewidth=0.08](13.401389,-0.5993054)(12.201389,-1.3993055)(12.201389,-1.3993055)
 			\psline[linecolor=black, linewidth=0.08](13.401389,-0.5993054)(12.201389,0.20069458)(12.201389,0.20069458)
 			\psline[linecolor=black, linewidth=0.08](13.401389,-2.9993055)(12.201389,-2.1993055)(12.201389,-2.1993055)
 			\psline[linecolor=black, linewidth=0.08](13.401389,-2.9993055)(12.201389,-2.9993055)(12.201389,-2.9993055)
 			\psline[linecolor=black, linewidth=0.08](13.401389,-2.9993055)(12.201389,-3.7993054)(12.201389,-3.7993054)
 			\psline[linecolor=black, linewidth=0.08](13.401389,-5.3993053)(12.201389,-4.5993056)(12.201389,-4.5993056)
 			\psline[linecolor=black, linewidth=0.08](13.401389,-5.3993053)(12.201389,-5.3993053)(12.201389,-5.3993053)
 			\psline[linecolor=black, linewidth=0.08](13.401389,-5.3993053)(12.201389,-6.1993055)(12.201389,-6.1993055)
 			\psdots[linecolor=black, dotstyle=o, dotsize=0.4, fillcolor=white](17.80139,0.20069458)
 			\psdots[linecolor=black, dotstyle=o, dotsize=0.4, fillcolor=white](17.80139,-0.5993054)
 			\psdots[linecolor=black, dotstyle=o, dotsize=0.4, fillcolor=white](17.80139,-1.3993055)
 			\psdots[linecolor=black, dotstyle=o, dotsize=0.4, fillcolor=white](17.80139,-2.1993055)
 			\psdots[linecolor=black, dotstyle=o, dotsize=0.4, fillcolor=white](17.80139,-2.9993055)
 			\psdots[linecolor=black, dotstyle=o, dotsize=0.4, fillcolor=white](17.80139,-3.7993054)
 			\psdots[linecolor=black, dotstyle=o, dotsize=0.4, fillcolor=white](17.80139,-4.5993056)
 			\psdots[linecolor=black, dotstyle=o, dotsize=0.4, fillcolor=white](17.80139,-5.3993053)
 			\psdots[linecolor=black, dotstyle=o, dotsize=0.4, fillcolor=white](17.80139,-6.1993055)
 			\psdots[linecolor=black, dotstyle=o, dotsize=0.4, fillcolor=white](12.201389,0.20069458)
 			\psdots[linecolor=black, dotstyle=o, dotsize=0.4, fillcolor=white](12.201389,-0.5993054)
 			\psdots[linecolor=black, dotstyle=o, dotsize=0.4, fillcolor=white](12.201389,-1.3993055)
 			\psdots[linecolor=black, dotstyle=o, dotsize=0.4, fillcolor=white](12.201389,-2.1993055)
 			\psdots[linecolor=black, dotstyle=o, dotsize=0.4, fillcolor=white](12.201389,-2.9993055)
 			\psdots[linecolor=black, dotstyle=o, dotsize=0.4, fillcolor=white](12.201389,-3.7993054)
 			\psdots[linecolor=black, dotstyle=o, dotsize=0.4, fillcolor=white](12.201389,-4.5993056)
 			\psdots[linecolor=black, dotstyle=o, dotsize=0.4, fillcolor=white](12.201389,-5.3993053)
 			\psdots[linecolor=black, dotstyle=o, dotsize=0.4, fillcolor=white](12.201389,-6.1993055)
 			\rput[bl](14.881389,-3.6793053){$w$}
 			\psdots[linecolor=black, dotstyle=o, dotsize=0.4, fillcolor=white](3.0013888,-2.9993055)
 			\psdots[linecolor=black, dotstyle=o, dotsize=0.4, fillcolor=white](4.601389,-2.9993055)
 			\psdots[linecolor=black, dotsize=0.4](13.401389,-0.5993054)
 			\psdots[linecolor=black, dotsize=0.4](13.401389,-2.9993055)
 			\psdots[linecolor=black, dotsize=0.4](13.401389,-5.3993053)
 			\psdots[linecolor=black, dotsize=0.4](16.601389,-0.5993054)
 			\psdots[linecolor=black, dotsize=0.4](16.601389,-2.9993055)
 			\psdots[linecolor=black, dotsize=0.4](16.601389,-5.3993053)
 			\psdots[linecolor=black, dotsize=0.4](15.001389,-2.9993055)
 			\psdots[linecolor=black, dotsize=0.4](1.4013889,-0.5993054)
 			\psdots[linecolor=black, dotsize=0.4](1.4013889,-2.9993055)
 			\psdots[linecolor=black, dotsize=0.4](1.4013889,-5.3993053)
 			\psdots[linecolor=black, dotsize=0.4](6.201389,-0.5993054)
 			\psdots[linecolor=black, dotsize=0.4](6.201389,-2.9993055)
 			\psdots[linecolor=black, dotsize=0.4](6.201389,-5.3993053)
 			\rput[bl](3.401389,-7.7993054){\LARGE{$G$}}
 			\rput[bl](14.601389,-7.9993052){\LARGE{$G/e$}}
 			\end{pspicture}
 		}
 	\end{center}
 	\caption{Graphs $G$ and $G/e$, respectively } \label{g/e-upper}
 \end{figure}
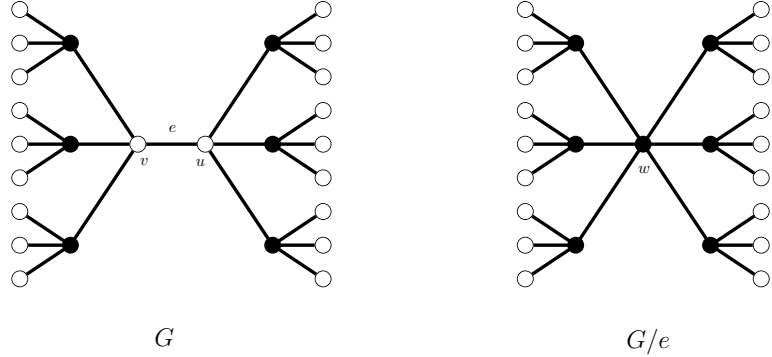

 \begin{remark}
 	{\normalfont
 		Bounds in Theorem \ref{edge-contraction} are tight. For the upper bound, consider
 		Figure \ref{g/e-upper}. One can easily check that the set of black vertices of $G$ and $G/e$  are strong dominating sets and we are done. For the lower bound, consider
 		Figure \ref{g/e-lower}. One can easily check that the set of black vertices of $H$ and $H/e$  are strong dominating sets, as desired. 
 	}
 \end{remark}

 \begin{figure}
 	\begin{center}
 		\psscalebox{0.56 0.56}
 		{
 			\begin{pspicture}(0,-7.9993057)(19.602777,-1.9979166)
 			\psline[linecolor=black, linewidth=0.08](3.8013887,-4.1993055)(5.4013886,-4.1993055)(5.4013886,-4.1993055)
 			\rput[bl](4.541389,-3.8993053){$e$}
 			\rput[bl](5.181389,-4.7193055){$u$}
 			\rput[bl](3.861389,-4.6993055){$v$}
 			\psline[linecolor=black, linewidth=0.08](5.4013886,-4.1993055)(6.601389,-2.9993055)(6.601389,-2.9993055)
 			\psline[linecolor=black, linewidth=0.08](5.4013886,-4.1993055)(6.601389,-5.3993053)(6.601389,-5.3993053)
 			\psline[linecolor=black, linewidth=0.08](3.8013887,-4.1993055)(2.601389,-2.9993055)(2.601389,-2.9993055)
 			\psline[linecolor=black, linewidth=0.08](3.8013887,-4.1993055)(2.601389,-5.3993053)(2.601389,-5.3993053)
 			\psline[linecolor=black, linewidth=0.08](6.601389,-2.9993055)(7.8013887,-2.1993055)(7.8013887,-2.1993055)
 			\psline[linecolor=black, linewidth=0.08](6.601389,-2.9993055)(7.8013887,-2.9993055)(7.8013887,-2.9993055)
 			\psline[linecolor=black, linewidth=0.08](6.601389,-2.9993055)(7.8013887,-3.7993054)(7.8013887,-3.7993054)
 			\psline[linecolor=black, linewidth=0.08](6.601389,-5.3993053)(7.8013887,-4.5993056)(7.8013887,-4.5993056)
 			\psline[linecolor=black, linewidth=0.08](6.601389,-5.3993053)(7.8013887,-5.3993053)(7.8013887,-5.3993053)
 			\psline[linecolor=black, linewidth=0.08](6.601389,-5.3993053)(7.8013887,-6.1993055)(7.8013887,-6.1993055)
 			\psline[linecolor=black, linewidth=0.08](2.601389,-2.9993055)(1.4013889,-2.1993055)(1.4013889,-2.1993055)
 			\psline[linecolor=black, linewidth=0.08](1.4013889,-2.9993055)(2.601389,-2.9993055)(2.601389,-2.9993055)
 			\psline[linecolor=black, linewidth=0.08](2.601389,-2.9993055)(1.4013889,-3.7993054)(1.4013889,-3.7993054)
 			\psline[linecolor=black, linewidth=0.08](2.601389,-5.3993053)(1.4013889,-4.5993056)(1.4013889,-4.5993056)
 			\psline[linecolor=black, linewidth=0.08](2.601389,-5.3993053)(1.4013889,-5.3993053)(1.4013889,-5.3993053)
 			\psline[linecolor=black, linewidth=0.08](2.601389,-5.3993053)(1.4013889,-6.1993055)(1.4013889,-6.1993055)
 			\psline[linecolor=black, linewidth=0.08](15.801389,-4.1993055)(17.001389,-2.9993055)(17.001389,-2.9993055)
 			\psline[linecolor=black, linewidth=0.08](15.801389,-4.1993055)(17.001389,-5.3993053)(17.001389,-5.3993053)
 			\psline[linecolor=black, linewidth=0.08](15.801389,-4.1993055)(14.601389,-2.9993055)(14.601389,-2.9993055)
 			\psline[linecolor=black, linewidth=0.08](15.801389,-4.1993055)(14.601389,-5.3993053)(14.601389,-5.3993053)
 			\psline[linecolor=black, linewidth=0.08](17.001389,-2.9993055)(18.20139,-2.1993055)(18.20139,-2.1993055)
 			\psline[linecolor=black, linewidth=0.08](17.001389,-2.9993055)(18.20139,-2.9993055)(18.20139,-2.9993055)
 			\psline[linecolor=black, linewidth=0.08](17.001389,-2.9993055)(18.20139,-3.7993054)(18.20139,-3.7993054)
 			\psline[linecolor=black, linewidth=0.08](17.001389,-5.3993053)(18.20139,-4.5993056)(18.20139,-4.5993056)
 			\psline[linecolor=black, linewidth=0.08](17.001389,-5.3993053)(18.20139,-5.3993053)(18.20139,-5.3993053)
 			\psline[linecolor=black, linewidth=0.08](17.001389,-5.3993053)(18.20139,-6.1993055)(18.20139,-6.1993055)
 			\psline[linecolor=black, linewidth=0.08](14.601389,-2.9993055)(13.401389,-2.1993055)(13.401389,-2.1993055)
 			\psline[linecolor=black, linewidth=0.08](13.401389,-2.9993055)(14.601389,-2.9993055)(14.601389,-2.9993055)
 			\psline[linecolor=black, linewidth=0.08](14.601389,-2.9993055)(13.401389,-3.7993054)(13.401389,-3.7993054)
 			\psline[linecolor=black, linewidth=0.08](14.601389,-5.3993053)(13.401389,-4.5993056)(13.401389,-4.5993056)
 			\psline[linecolor=black, linewidth=0.08](14.601389,-5.3993053)(13.401389,-5.3993053)(13.401389,-5.3993053)
 			\psline[linecolor=black, linewidth=0.08](14.601389,-5.3993053)(13.401389,-6.1993055)(13.401389,-6.1993055)
 			\rput[bl](15.641389,-4.7193055){$w$}
 			\rput[bl](4.201389,-7.7993054){\LARGE{$H$}}
 			\rput[bl](15.401389,-7.9993052){\LARGE{$H/e$}}
 			\psline[linecolor=black, linewidth=0.08](7.8013887,-2.1993055)(9.001389,-2.1993055)(9.001389,-2.1993055)
 			\psline[linecolor=black, linewidth=0.08](7.8013887,-2.9993055)(9.001389,-2.9993055)(9.001389,-2.9993055)
 			\psline[linecolor=black, linewidth=0.08](7.8013887,-3.7993054)(9.001389,-3.7993054)(9.001389,-3.7993054)
 			\psline[linecolor=black, linewidth=0.08](7.8013887,-4.5993056)(9.001389,-4.5993056)(9.001389,-4.5993056)
 			\psline[linecolor=black, linewidth=0.08](7.8013887,-5.3993053)(9.001389,-5.3993053)(9.001389,-5.3993053)
 			\psline[linecolor=black, linewidth=0.08](7.8013887,-6.1993055)(9.001389,-6.1993055)(9.001389,-6.1993055)
 			\psline[linecolor=black, linewidth=0.08](18.20139,-2.1993055)(19.401388,-2.1993055)(19.401388,-2.1993055)
 			\psline[linecolor=black, linewidth=0.08](18.20139,-2.9993055)(19.401388,-2.9993055)(19.401388,-2.9993055)
 			\psline[linecolor=black, linewidth=0.08](18.20139,-3.7993054)(19.401388,-3.7993054)(19.401388,-3.7993054)
 			\psline[linecolor=black, linewidth=0.08](18.20139,-4.5993056)(19.401388,-4.5993056)(19.401388,-4.5993056)
 			\psline[linecolor=black, linewidth=0.08](18.20139,-5.3993053)(19.401388,-5.3993053)(19.401388,-5.3993053)
 			\psline[linecolor=black, linewidth=0.08](18.20139,-6.1993055)(19.401388,-6.1993055)(19.401388,-6.1993055)
 			\psline[linecolor=black, linewidth=0.08](12.201389,-2.1993055)(13.401389,-2.1993055)(13.401389,-2.1993055)
 			\psline[linecolor=black, linewidth=0.08](12.201389,-2.9993055)(13.401389,-2.9993055)(13.401389,-2.9993055)
 			\psline[linecolor=black, linewidth=0.08](12.201389,-3.7993054)(13.401389,-3.7993054)(13.401389,-3.7993054)
 			\psline[linecolor=black, linewidth=0.08](12.201389,-4.5993056)(13.401389,-4.5993056)(13.401389,-4.5993056)
 			\psline[linecolor=black, linewidth=0.08](12.201389,-5.3993053)(13.401389,-5.3993053)(13.401389,-5.3993053)
 			\psline[linecolor=black, linewidth=0.08](12.201389,-6.1993055)(13.401389,-6.1993055)(13.401389,-6.1993055)
 			\psline[linecolor=black, linewidth=0.08](0.20138885,-2.1993055)(1.4013889,-2.1993055)(1.4013889,-2.1993055)
 			\psline[linecolor=black, linewidth=0.08](0.20138885,-2.9993055)(1.4013889,-2.9993055)(1.4013889,-2.9993055)
 			\psline[linecolor=black, linewidth=0.08](0.20138885,-3.7993054)(1.4013889,-3.7993054)(1.4013889,-3.7993054)
 			\psline[linecolor=black, linewidth=0.08](0.20138885,-4.5993056)(1.4013889,-4.5993056)(1.4013889,-4.5993056)
 			\psline[linecolor=black, linewidth=0.08](0.20138885,-5.3993053)(1.4013889,-5.3993053)(1.4013889,-5.3993053)
 			\psline[linecolor=black, linewidth=0.08](0.20138885,-6.1993055)(1.4013889,-6.1993055)(1.4013889,-6.1993055)
 			\psdots[linecolor=black, dotsize=0.4](15.801389,-4.1993055)
 			\psdots[linecolor=black, dotsize=0.4](13.401389,-2.1993055)
 			\psdots[linecolor=black, dotsize=0.4](13.401389,-2.9993055)
 			\psdots[linecolor=black, dotsize=0.4](13.401389,-3.7993054)
 			\psdots[linecolor=black, dotsize=0.4](13.401389,-4.5993056)
 			\psdots[linecolor=black, dotsize=0.4](13.401389,-5.3993053)
 			\psdots[linecolor=black, dotsize=0.4](13.401389,-6.1993055)
 			\psdots[linecolor=black, dotsize=0.4](18.20139,-2.1993055)
 			\psdots[linecolor=black, dotsize=0.4](18.20139,-2.9993055)
 			\psdots[linecolor=black, dotsize=0.4](18.20139,-3.7993054)
 			\psdots[linecolor=black, dotsize=0.4](18.20139,-4.5993056)
 			\psdots[linecolor=black, dotsize=0.4](18.20139,-5.3993053)
 			\psdots[linecolor=black, dotsize=0.4](18.20139,-6.1993055)
 			\psdots[linecolor=black, dotsize=0.4](6.601389,-2.9993055)
 			\psdots[linecolor=black, dotsize=0.4](6.601389,-5.3993053)
 			\psdots[linecolor=black, dotsize=0.4](2.601389,-2.9993055)
 			\psdots[linecolor=black, dotsize=0.4](2.601389,-5.3993053)
 			\psdots[linecolor=black, dotsize=0.4](1.4013889,-2.1993055)
 			\psdots[linecolor=black, dotsize=0.4](1.4013889,-2.9993055)
 			\psdots[linecolor=black, dotsize=0.4](1.4013889,-3.7993054)
 			\psdots[linecolor=black, dotsize=0.4](1.4013889,-4.5993056)
 			\psdots[linecolor=black, dotsize=0.4](1.4013889,-5.3993053)
 			\psdots[linecolor=black, dotsize=0.4](1.4013889,-6.1993055)
 			\psdots[linecolor=black, dotsize=0.4](7.8013887,-2.1993055)
 			\psdots[linecolor=black, dotsize=0.4](7.8013887,-2.9993055)
 			\psdots[linecolor=black, dotsize=0.4](7.8013887,-3.7993054)
 			\psdots[linecolor=black, dotsize=0.4](7.8013887,-4.5993056)
 			\psdots[linecolor=black, dotsize=0.4](7.8013887,-5.3993053)
 			\psdots[linecolor=black, dotsize=0.4](7.8013887,-6.1993055)
 			\psdots[linecolor=black, dotstyle=o, dotsize=0.4, fillcolor=white](12.201389,-2.1993055)
 			\psdots[linecolor=black, dotstyle=o, dotsize=0.4, fillcolor=white](12.201389,-2.9993055)
 			\psdots[linecolor=black, dotstyle=o, dotsize=0.4, fillcolor=white](12.201389,-3.7993054)
 			\psdots[linecolor=black, dotstyle=o, dotsize=0.4, fillcolor=white](12.201389,-4.5993056)
 			\psdots[linecolor=black, dotstyle=o, dotsize=0.4, fillcolor=white](12.201389,-5.3993053)
 			\psdots[linecolor=black, dotstyle=o, dotsize=0.4, fillcolor=white](12.201389,-6.1993055)
 			\psdots[linecolor=black, dotstyle=o, dotsize=0.4, fillcolor=white](14.601389,-2.9993055)
 			\psdots[linecolor=black, dotstyle=o, dotsize=0.4, fillcolor=white](14.601389,-5.3993053)
 			\psdots[linecolor=black, dotstyle=o, dotsize=0.4, fillcolor=white](17.001389,-2.9993055)
 			\psdots[linecolor=black, dotstyle=o, dotsize=0.4, fillcolor=white](17.001389,-5.3993053)
 			\psdots[linecolor=black, dotstyle=o, dotsize=0.4, fillcolor=white](19.401388,-2.1993055)
 			\psdots[linecolor=black, dotstyle=o, dotsize=0.4, fillcolor=white](19.401388,-2.9993055)
 			\psdots[linecolor=black, dotstyle=o, dotsize=0.4, fillcolor=white](19.401388,-3.7993054)
 			\psdots[linecolor=black, dotstyle=o, dotsize=0.4, fillcolor=white](19.401388,-4.5993056)
 			\psdots[linecolor=black, dotstyle=o, dotsize=0.4, fillcolor=white](19.401388,-5.3993053)
 			\psdots[linecolor=black, dotstyle=o, dotsize=0.4, fillcolor=white](19.401388,-6.1993055)
 			\psdots[linecolor=black, dotstyle=o, dotsize=0.4, fillcolor=white](9.001389,-2.1993055)
 			\psdots[linecolor=black, dotstyle=o, dotsize=0.4, fillcolor=white](9.001389,-2.9993055)
 			\psdots[linecolor=black, dotstyle=o, dotsize=0.4, fillcolor=white](9.001389,-3.7993054)
 			\psdots[linecolor=black, dotstyle=o, dotsize=0.4, fillcolor=white](9.001389,-4.5993056)
 			\psdots[linecolor=black, dotstyle=o, dotsize=0.4, fillcolor=white](9.001389,-5.3993053)
 			\psdots[linecolor=black, dotstyle=o, dotsize=0.4, fillcolor=white](9.001389,-6.1993055)
 			\psdots[linecolor=black, dotstyle=o, dotsize=0.4, fillcolor=white](5.4013886,-4.1993055)
 			\psdots[linecolor=black, dotstyle=o, dotsize=0.4, fillcolor=white](3.8013887,-4.1993055)
 			\psdots[linecolor=black, dotstyle=o, dotsize=0.4, fillcolor=white](0.20138885,-2.1993055)
 			\psdots[linecolor=black, dotstyle=o, dotsize=0.4, fillcolor=white](0.20138885,-2.9993055)
 			\psdots[linecolor=black, dotstyle=o, dotsize=0.4, fillcolor=white](0.20138885,-3.7993054)
 			\psdots[linecolor=black, dotstyle=o, dotsize=0.4, fillcolor=white](0.20138885,-4.5993056)
 			\psdots[linecolor=black, dotstyle=o, dotsize=0.4, fillcolor=white](0.20138885,-5.3993053)
 			\psdots[linecolor=black, dotstyle=o, dotsize=0.4, fillcolor=white](0.20138885,-6.1993055)
 			\end{pspicture}
 		}
 	\end{center}
 	\caption{Graphs $H$ and $H/e$, respectively } \label{g/e-lower}
 \end{figure}
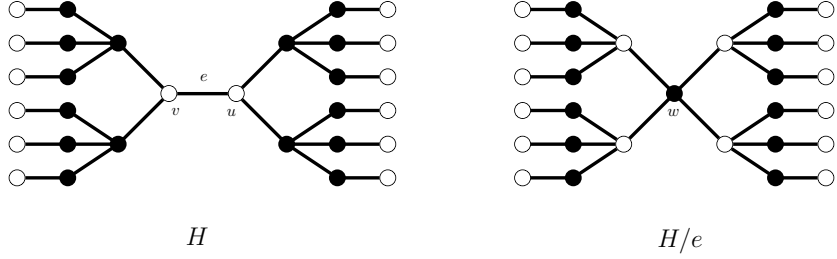
 
 \begin{remark}
 {\normalfont
 The left equality in the Theorem \ref{edge-contraction}  is true for the cycles $C_{3k+1}$.  	
 }
  \end{remark}

 As an immediate result of Theorems  \ref{edge-deletion}, \ref{edge-subdivision}, and \ref{edge-contraction}, we have:

 \begin{corollary}
 	Let $\alpha=\gamma_{st}(G-e)+\gamma_{st}(G_e)+\gamma_{st}(G/e)$, and $\beta=\deg(u)+\deg(v)$. Then,
 	$$\frac{\alpha-\beta}{3}\leq\gamma_{st}(G)\leq \frac{\alpha+\beta+2}{3}.$$
 \end{corollary}

 \section{Strong domination number of $k$-subdivision of a graph}

The { $k$-subdivision} of $G$, denoted by $G^{\frac{1}{k}}$, is constructed by replacing each edge $v_iv_j$ of $G$ with a path of length $k$, say $P^{\{v_i,v_j\}}$. These $k$-paths are called \textit{superedges}, any new vertex is an internal vertex, and is denoted by $x^{\{v_i,v_j\}}_l$ if it belongs to the superedge $P_{\{v_i,v_j\}}$, $i<j$ with  distance $l$ from the vertex $v_i$, where $l \in \{1, 2, \ldots , k-1\}$ (see for example Figure \ref{2-sub}).  Note that for $k = 1$, we have $G^{1/1}= G^1 = G$, and if  $G$ has $n$ vertices and $m$ edges, then the graph $G^{\frac{1}{k}}$ has $n+(k-1)m$ vertices and $km$ edges. Some results about subdivision of a graph can be found in  \cite{KSUB,Babu,Nima2}. In this section, we study the strong domination number of $k$-subdivision of a graph.  First, we consider the graphs with minimum degree at least $3$.

 \begin{figure}
 	\begin{center}
 		\psscalebox{0.55 0.55}
{
\begin{pspicture}(0,-8.225)(17.48,-0.075)
\psline[linecolor=black, linewidth=0.08](2.16,-0.725)(6.16,-0.725)(7.36,-3.525)(6.16,-6.325)(2.16,-6.325)(0.96,-3.525)(2.16,-0.725)(2.16,-0.725)
\psline[linecolor=black, linewidth=0.08](2.16,-0.725)(6.16,-6.325)(6.16,-6.325)
\psline[linecolor=black, linewidth=0.08](7.36,-3.525)(0.96,-3.525)(0.96,-3.525)
\psline[linecolor=black, linewidth=0.08](6.16,-0.725)(2.16,-6.325)(2.56,-6.325)
\psdots[linecolor=black, dotsize=0.4](2.16,-0.725)
\psdots[linecolor=black, dotsize=0.4](6.16,-0.725)
\psdots[linecolor=black, dotsize=0.4](7.36,-3.525)
\psdots[linecolor=black, dotsize=0.4](6.16,-6.325)
\psdots[linecolor=black, dotsize=0.4](4.16,-3.525)
\psdots[linecolor=black, dotsize=0.4](0.96,-3.525)
\psdots[linecolor=black, dotsize=0.4](2.16,-6.325)
\rput[bl](1.58,-0.325){$u_1$}
\rput[bl](6.28,-0.345){$u_2$}
\rput[bl](7.88,-3.665){$u_3$}
\rput[bl](6.5,-6.945){$u_4$}
\rput[bl](1.72,-6.925){$u_5$}
\rput[bl](0.0,-3.665){$u_6$}
\rput[bl](4.0,-3.005){$u_7$}
\psline[linecolor=black, linewidth=0.08](11.36,-0.725)(15.36,-0.725)(16.56,-3.525)(15.36,-6.325)(11.36,-6.325)(10.16,-3.525)(11.36,-0.725)(11.36,-0.725)
\psline[linecolor=black, linewidth=0.08](11.36,-0.725)(15.36,-6.325)(15.36,-6.325)
\psline[linecolor=black, linewidth=0.08](16.56,-3.525)(10.16,-3.525)(10.16,-3.525)
\psline[linecolor=black, linewidth=0.08](15.36,-0.725)(11.36,-6.325)(11.76,-6.325)
\psdots[linecolor=black, dotsize=0.4](11.36,-0.725)
\psdots[linecolor=black, dotsize=0.4](15.36,-0.725)
\psdots[linecolor=black, dotsize=0.4](16.56,-3.525)
\psdots[linecolor=black, dotsize=0.4](15.36,-6.325)
\psdots[linecolor=black, dotsize=0.4](13.36,-3.525)
\psdots[linecolor=black, dotsize=0.4](10.16,-3.525)
\psdots[linecolor=black, dotsize=0.4](11.36,-6.325)
\rput[bl](10.78,-0.325){$u_1$}
\rput[bl](15.48,-0.345){$u_2$}
\rput[bl](17.08,-3.665){$u_3$}
\rput[bl](15.7,-6.945){$u_4$}
\rput[bl](10.92,-6.925){$u_5$}
\rput[bl](9.2,-3.665){$u_6$}
\psdots[linecolor=black, fillstyle=solid,fillcolor=black, dotsize=0.4](13.36,-0.725)
\psdots[linecolor=black, fillstyle=solid,fillcolor=black, dotsize=0.4](13.36,-6.325)
\psdots[linecolor=black, fillstyle=solid,fillcolor=black, dotsize=0.4](11.76,-3.525)
\psdots[linecolor=black, fillstyle=solid,fillcolor=black, dotsize=0.4](14.96,-3.525)
\psdots[linecolor=black, fillstyle=solid,fillcolor=black, dotsize=0.4](14.38,-2.085)
\psdots[linecolor=black, fillstyle=solid,fillcolor=black, dotsize=0.4](12.34,-2.085)
\psdots[linecolor=black, fillstyle=solid,fillcolor=black, dotsize=0.4](15.94,-2.145)
\psdots[linecolor=black, fillstyle=solid,fillcolor=black, dotsize=0.4](15.98,-4.945)
\psdots[linecolor=black, fillstyle=solid,fillcolor=black, dotsize=0.4](14.32,-4.925)
\psdots[linecolor=black, fillstyle=solid,fillcolor=black, dotsize=0.4](12.34,-4.945)
\psdots[linecolor=black, fillstyle=solid,fillcolor=black, dotsize=0.4](10.76,-4.905)
\psdots[linecolor=black, fillstyle=solid,fillcolor=black, dotsize=0.4](10.78,-2.105)
\rput[bl](13.32,-8.225){$G^{\frac{1}{2}}$}
\rput[bl](3.98,-8.165){$G$}
\rput[bl](13.2,-3.005){$u_7$}
\end{pspicture}
}
 	\end{center}
 	\caption{Graphs $G$ and $G^{\frac{1}{2}}$, respectively} \label{2-sub}
 \end{figure}
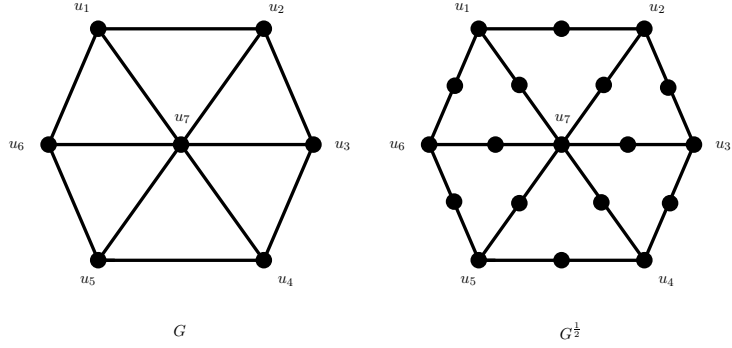

   \begin{theorem}\label{k-sub}
   	Let $G$ be a graph of order $n$, size $m$, and $\delta (G)\geq 3$. Then, 
   	\begin{equation*}
   	\gamma_{st}(G^{\frac{1}{k}})=\left\{
   	\begin{array}{ll}
   	n  &\quad\mbox{if  $k=2,3$},\\
   	n+m \left\lceil \frac{k-3}{3} \right\rceil   &\quad\mbox{otherwise}.
   	\end{array}\right.
   	\end{equation*}
   \end{theorem}

   \begin{proof}
Suppose that $v_iv_j\in E(G)$. First, let $k=2$. Then, $P^{\{v_i,v_j\}}$ consists of vertices $v_i$, $x^{\{v_i,v_j\}}_1$, and $v_j$. Since $\deg (x^{\{v_i,v_j\}}_1)=2$ and $\delta (G)\geq 3$, then we should have $v_i$ and $v_j$ in our strong dominating set. Hence, $\gamma_{st}(G^{\frac{1}{2}})=n$. By the same argument, we have $\gamma_{st}(G^{\frac{1}{3}})=n$, too. Now consider the graph $G^{\frac{1}{k}}$, where $k\geq 4$. Then, $P^{\{v_i,v_j\}}$ consists of vertices 
$v_i, x^{\{v_i,v_j\}}_1, x^{\{v_i,v_j\}}_2,\ldots, x^{\{v_i,v_j\}}_{k-1}, v_j$.
By the same argument as cases $k=2,3$, we need $v_i$ and $v_j$ in our strong dominating set, and they  strong dominate vertices $x^{\{v_i,v_j\}}_1$ and $x^{\{v_i,v_j\}}_{k-1}$, respectively. Now, for the rest of vertices, we have a path  of order $k-3$, and since we need $\left\lceil \frac{k-3}{3} \right\rceil$ vertices among them to have a strong dominating set for this path, then the proof is complete.
   	\qed
   \end{proof}

 By the same argument as proof of Theorem \ref{k-sub}, we have the upper bound in case $\delta (G)\geq 2$. 
 
    \begin{theorem}\label{k-sub-2}
   	Let $G$ be a graph of order $n$, size $m$, and $\delta (G)\geq 2$.  Then, 
   	\begin{equation*}
   	\gamma_{st}(G^{\frac{1}{k}})\leq\left\{
   	\begin{array}{ll}
   	n  &\quad\mbox{if  $k=2,3$},\\
   	n+m \left\lceil \frac{k-3}{3} \right\rceil   &\quad\mbox{otherwise}.
   	\end{array}\right.
   	\end{equation*}
   \end{theorem}
 
The following example shows that for some graphs and some $k\in\mathbb{N}\setminus\{1\}$, the equality holds, and for some it does not.

 \begin{example}
 	{\normalfont
Let $G=C_5$. Then one can easily check that $	\gamma_{st}(G^{\frac{1}{2}})=4<5$, and $	\gamma_{st}(G^{\frac{1}{k}})<n ( 1+{\left\lceil \frac{k-3}{3} \right\rceil}) $, where $k \in\mathbb{N}\setminus\{1,2,3t ~| ~t\in \mathbb{N}\}$. But, $\gamma_{st}(G^{\frac{1}{3r}})=nr$, where $r\in \mathbb{N}$, as desired.
 	}
 \end{example}

Now, we consider graphs with pendant vertices and find an upper bound for $\gamma_{st}(G^{\frac{1}{k}})$.

    \begin{theorem}\label{k-sub-3}
   	Let $G$ be a graph of order $n$, size $m$, and $t$ pendant vertices, where $1\leq t\leq n-1$.  Then, 
   	\begin{equation*}
   	\gamma_{st}(G^{\frac{1}{k}})\leq\left\{
   	\begin{array}{ll}
   	n  &\quad\mbox{if  $k=2,3$},\\
   	n+ t\left\lceil \frac{k-4}{3} \right\rceil + (m-t) \left\lceil \frac{k-3}{3} \right\rceil   &\quad\mbox{otherwise}.
   	\end{array}\right.
   	\end{equation*}
   \end{theorem}

   \begin{proof}
Suppose that $v_iv_j\in E(G)$, and $v_i$ is a pendant vertex. First, let $k=2$. Then, $P^{\{v_i,v_j\}}$ consists of vertices $v_i$, $x^{\{v_i,v_j\}}_1$, and $v_j$. Since $\deg (x^{\{v_i,v_j\}}_1)=2$ and $\deg (v_i)=1$, then we should have $x^{\{v_i,v_j\}}_1$ in our strong dominating set. So the set $S$ containing  these vertices and non-pendant vertices of $G$, is a strong dominating set and we are done.  By the same argument, we have $\gamma_{st}(G^{\frac{1}{3}})\leq n$, too. Now consider the graph $G^{\frac{1}{k}}$, where $k\geq 4$. The superedge  $P^{\{v_i,v_j\}}$ consists of vertices 
$v_i, x^{\{v_i,v_j\}}_1, x^{\{v_i,v_j\}}_2,\ldots, x^{\{v_i,v_j\}}_{k-1}, v_j$.
By the same argument as cases $k=2,3$, we pick  $x^{\{v_i,v_j\}}_1$ and $v_j$ in our strong dominating set, and they  strong dominate vertices $v_i$ and $x^{\{v_i,v_j\}}_{k-1}$, respectively. Now, for the rest of vertices of $P^{\{v_i,v_j\}}$, we have a path graph of order $k-4$, and since we need $\left\lceil \frac{k-4}{3} \right\rceil$ vertices among them to have a strong dominating set for this path, then by adding cases when we do not have a pendant vertex as endpoint of an edge (same argument as proof of Theorem \ref{k-sub}), we have the result. 
   	\qed
   \end{proof}

 \begin{remark}
 {\normalfont
 The upper bound in the Theorem \ref{k-sub-3}  is tight, if $k\equiv 0$ $(mod\,3)$.  It suffices to consider $G$ as the path graph $P_4$.
 }
  \end{remark}

 \end{document}